\documentclass[runningheads]{llncs}
\usepackage{graphicx}
\usepackage{appendix}
\usepackage{amsmath}
\usepackage{amssymb}
\usepackage{amsfonts}
\usepackage{hyperref}
\usepackage{tikz-cd}
\usepackage{amsthm}
\usepackage{algorithmic}
\usepackage{graphicx}
\usepackage[ruled,vlined]{algorithm2e}
\usepackage{tabu}
\theoremstyle{definition}
\usepackage{mathtools}

\newtheorem{defn}{Definition}

\titlerunning{On the Commuting Probability of Finite Rings}
\begin{document}
\title{A Note On Square-free Commuting Probabilities of Finite Rings}
%
%
\author{Andrew Mendelsohn}
\authorrunning{A. Mendelsohn}
%
\institute{Department of EEE, Imperial College London, London, SW7 2AZ, United Kingdom. \\ \href{andrew.mendelsohn18@imperial.ac.uk}{andrew.mendelsohn18@imperial.ac.uk}}
\maketitle
\begin{abstract}
It is shown that the commuting probability of a finite ring cannot be a fraction with square-free denominator, resolving a conjecture of Buckley and MacHale.
\end{abstract}

\section{Introduction}
Let $G$ be a finite group and $U$ denote the uniform distribution. Let the commuting probability of $G$, $P_G$, be denoted $$P_G:=Pr_{a,b\leftarrow U(G)}(ab=ba).$$ An alternative characterisation is $$P_G = \frac{1}{|G|^2}|\{(x,y)\in G^2\text{ : }xy = yx\}|.$$
Joseph made the following conjectures, where $\mathcal{G}$ is the set of all commuting probabilities of finite groups \cite{josephconjectures}:
\begin{enumerate}
\item[$\text{A}.$] All limit points of $\mathcal{G}$ are rational.
\item[$\text{B}.$] The set $\mathcal{G}$ is well ordered by $>$.
\item[$\text{C}.$] The set ${0}\cup\mathcal{G}$ is closed (that is, contains all its accumulation points).
\end{enumerate}
In \cite{eberhard}, Eberhard resolved conjectures A and B, and in \cite{limitpointscommprob} conjecture C was resolved.\\
One can extend the definition of commuting probability to finite rings: set $$P_R = \frac{1}{|R|^2}|\{(x,y)\in R^2\text{ : }xy = yx\}|=\frac{1}{|R|^2}|\{(x,y)\in R^2\text{ : }xy - yx=0\}|.$$
In \cite{contrastingprobs}, the following conjectures were made, where $\mathcal{R}$ is the set $P_R$ over all finite rings $R$:
\begin{enumerate}
    \item $1/n \not\in \mathcal{R}$ when $n \in\mathbb{N}$ is square-free.
    \item $\mathcal{R}\subset\mathcal{G}$.
    \item $\mathcal{R}$ coincides with the set of values of $P_G$ as $G$ ranges over all finite
nilpotent groups of class at most 2.
    \item All limit points of $\mathcal{R}$ are rational.
    \item For each $0 < t \leq 1$, there exists $\epsilon_t > 0$ such that $\mathcal{R} \cap (t - \epsilon_t, t) = \emptyset$.
\item $\mathcal{R}$ does not contain any of its accumulation points.
\end{enumerate}
Note conjectures 4, 5, and 6 correspond to Joseph's conjectures. Moreover, conjectures 4 and 5 would follow from the veracity of conjectures 2 or 3, since Eberhard showed $\mathcal{G}$ has rational limit points and is well-ordered. In \cite{jurasursul}, conjecture 2 was in fact resolved, and thus conjectures 4 and 5. Moreover, conjecture 3 was partially resolved: the authors obtained that $\mathcal{R}$ is a subset of the set of values of $P_G$ as $G$ ranges over all finite nilpotent groups of class at most 2. We conclude that conjectures 1, 3, and 6 are open.\\
\indent In this work, we resolve conjecture 1.
\section{Preliminaries and Prior Results}
\begin{defn}
Two finite groups $G,H$ are called \textit{isoclinic} if $G/Z(G)\cong H/Z(H)$ and $G^\prime\cong H^\prime$, and if the diagram below commutes:
$$
\begin{tikzcd}
{G}/{Z(G)} \times {G}/{Z(G)} \arrow[d] \arrow[r] & {H}/{Z(H)} \times {H}/{Z(H)} \arrow[d] \\
G^\prime \arrow[r]                               & H^\prime                              
\end{tikzcd}
$$
\end{defn}
Isoclinism preserves nilpotency class and commuting probability \cite{lescott}. A \textit{stem group} is a group in a given isoclinism class of minimal order. It is well known that if $G$ is a stem group, then $Z(G)\leq G^\prime$. For more on isoclinism, see \cite{gpspordervol1}.\\
\indent Below we state existing results in the literature we will need below.
\begin{lemma}\cite{jurasursul}\label{jurasursul}
$\mathcal{R}\subset \mathcal{G}_{n,2}$, where $\mathcal{G}_{n,2}$ is the set of commuting probabilities of all finite nilpotent groups of class at most 2.
\end{lemma}
This statement is proved as follows: let $R$ be a finite ring. We can turn $R\oplus R$ into a nilpotent ring of class 3 by endowing it with the multiplication rule $(a,x)(b,y)=(0,ab)$. This ring can be turned into a nilpotent group $G_R$ of class at most 2 by endowing it with the binary operation $a\circ b = a + b + ab$. Both of these transformations preserve the commuting probability. Thus the values of $\mathcal{R}\setminus\{1\}$ are a subset of the values $P_G$, running over nilpotent groups $G$ of class equal to 2. Note that if $R$ has size $n$, then the resulting group $G_R$ has order $n^2$, and if $R$ is noncommutative then the resulting group is nonabelian.
\begin{lemma}\cite{lowerbdcommprob}\label{nilpprob}
$\operatorname{Pr}(G)=\frac{1}{\left|G^{\prime}\right|}\left(1+\frac{\left|G^{\prime}\right|-1}{|G: Z(G)|}\right)$ if and only if $G$ is nilpotent.
\end{lemma}
\begin{lemma}\cite{castelaz}\label{1/p}
If $G$ is a nilpotent group, then $P_G\ne\frac{1}{p}$.
\end{lemma}

\section{Results}
\begin{theorem}
$\frac{1}{p}\not\in\mathcal{R}$ for all $p\in\mathbb{N}_{\geq 2}$.
\end{theorem}
\begin{proof}
By Lemma \ref{jurasursul}, $\mathcal{R}$ is contained within the set of commuting probabilities of finite nilpotent groups of class at most 2. By Lemma \ref{1/p}, this latter set does not contain $\frac{1}{p}$ for any prime $p$.
\end{proof}
Denote the set of commuting probabilities of rings of prime power order for some prime $p$ by $\mathcal{R}_p$.
\begin{proposition}
$\frac{1}{n}\not\in\mathcal{R}_p$ for any prime $p\in\mathbb{N}_{\geq 2}$ and $n\in\mathbb{N}_{>1}$.
\end{proposition}
\begin{proof}
By Lemma \ref{jurasursul}, we need only consider commuting probabilities of finite nilpotent groups of class at most 2. By Lemma \ref{nilpprob}, we know $\textit{a fortiori}$ the commuting probability of finite nilpotent groups of class at most 2 in terms of derived subgroups and centers. Suppose that for some $n\in\mathbb{N}_{\geq2}$ we have $$\frac{1}{\left|G^{\prime}\right|}\left(1+\frac{\left|G^{\prime}\right|-1}{|G: Z(G)|}\right) = \frac{1}{n}.$$ By the construction of \cite{jurasursul} considered above, wlog let $|G| = p^{e}$ for some even positive integer $e$. Then $Z(G) = p^{f}$ and $G^\prime = p^{g}$ with $0<g\leq f<e$ (since $G$ is at most class 2). Then 
\begin{align} 
\frac{1}{n} &= p^{-g}\left(1+\frac{p^{g}-1}{p^{e-f}} \right) = p^{-g}+\frac{p^{g}-1}{p^{e-f+g}}\\ &= \frac{p^{e-f}+p^{g}-1}{p^{e-f+g}}.
\end{align}
For this to hold, some power $p^{h}$ must divide the numerator, with $h>0$; but this cannot hold; for if so, then one must have $p^{e-f}+p^{g}-1 = kp$, for some integer $k\ne0$. But then $-1\equiv0\bmod p$, a contradiction.
\end{proof}

\begin{theorem}
$\frac{\ell}{n}\not\in\mathcal{R}$ for any squarefree $n\in\mathbb{N}_{>1}$ and $\ell<n$ with $\operatorname{gcd}(\ell,n)=1$.
\end{theorem}
\begin{proof}
Any finite ring can be turned into a nilpotent group of class at most 2, such that the commuting probability of the ring is equal to the commuting probability of the group. The construction (outlined above) turns a commutative ring into a group of class 1, and a noncommutative ring into a nonabelian group of class at most 2, therefore of class equal to 2. The order of the group is the square of the order of the ring, so the Sylow subgroups of the group have order at least the square of a prime. Since the group is nilpotent, it can be written as a product of its Sylow subgroups, which are all of class at most 2, and the commuting probability of the group is the product of the commuting probabilities of its Sylow subgroups. Thus it remains to analyse the equation 
$$\frac{\ell}{n} = \prod_{i=1}^m\frac{p_i^{e_i-f_i}+p_i^{g_i}-1}{p_i^{e_i-f_i+g_i}},$$ for $m>1$ and the $p_i$ distinct, where the $e_i, f_i$, and $g_i$ are as before. Via isoclinism, we may replace $G_R$ by a class two nilpotent (stem) group $G$ with identical commuting probability and minimal order. Thus we may assume that $Z(G_R) = G_R^\prime$ (note isoclinism preserves nilpotency class and $G_R$ is class two), and moreover that none of the Sylow subgroups are abelian. The above equality simplifies to 
$$\frac{\ell}{n} = \prod_{i=1}^m\frac{p_i^{e^\prime_i-f^\prime_i}+p_i^{f^\prime_i}-1}{p_i^{e^\prime_i}},$$ where the exponents $e_i^\prime, f_i^\prime$ correspond to the group $G$. We now proceed by induction on the number of prime factors of $|G_R|$, denoted $m$.\\
\indent By Lemma 14 of \cite{prob1/3}, if $P_{G_R}=\frac{\ell}{n}$ in lowest terms, the prime factors of $n$ are precisely the prime factors of $|G_R|$. If $m=1$, it is known that $\frac{\ell}{n}\not\in \mathcal{R}_q$ for any prime $q$ and square-free $n$ (in fact, we know this to hold for $m\leq 69$ by Theorem 9 of \cite{contrastingprobs}). Suppose the statement is true up to $n=k-1$, and consider the case $n=k$; suppose, for a contradiction, that the commuting probability is equal to $\frac{\ell}{n}$, for some square-free integer $n$ and $\ell\leq n$ with $\operatorname{gcd}(\ell,n)=1$, and without loss of generality that $n$ has prime factors equal to the set of $p_i$, $i=1,...,k$: 
$$\frac{\ell}{n} = \prod_{i=1}^k\frac{p_i^{e^\prime_i-f^\prime_i}+p_i^{f^\prime_i}-1}{p_i^{e^\prime_i}}.$$
Rearrange for the following:
$$\frac{\ell\cdot p_k^{e_k^\prime}}{n\cdot (p_k^{e^\prime_k-f^\prime_k}+p_k^{f^\prime_k}-1)} = \prod_{i=1}^{k-1}\frac{p_i^{e^\prime_i-f^\prime_i}+p_i^{f^\prime_i}-1}{p_i^{e^\prime_i}}.$$
Writing the left hand side in lowest terms, we have
$$\frac{\ell\cdot p_k^{e_k^{\prime\prime}}}{n^\prime\cdot (p_k^{e^\prime_k-f^\prime_k}+p_k^{f^\prime_k}-1)} = \prod_{i=1}^{k-1}\frac{p_i^{e^\prime_i-f^\prime_i}+p_i^{f^\prime_i}-1}{p_i^{e^\prime_i}},$$ where $n^\prime$ is not divisible by $p_k$. We have a commuting probability on the right hand side with $k-1$ prime factors; so by the induction hypothesis, the denominator of the left hand side has no square factors. But we also have $p_k^{e^\prime_k-f^\prime_k}+p_k^{f^\prime_k}-1$ on the denominator of the left hand side, which is not divisible by $p_k$, and by Lemma 14 of \cite{prob1/3} must have prime factors equal to the set of $p_i$, for $i=1,...,k-1$; moreover, there can be no cancellation between these factors and $\ell$, by assumption on $\ell$. But then for at least one index $j$, $n^\prime\cdot (p_k^{e^\prime_k-f^\prime_k}+p_k^{f^\prime_k}-1)$ has a prime factor $p_j$ with multiplicity at least two, which is a contradiction.
\end{proof}

\begin{remark}
Since $\frac{1}{p}$ is an accumulation point of $\mathcal{R}_p$, $\frac{1}{n}$ is an accumulation point of $\mathcal{R}$ for all $n$. The above result thus means that many accumulation points of $\mathcal{R}$ are not contained in $\mathcal{R}$. As well as resolving the first conjecture stated at the beginning of this note, the result also makes progress on the sixth conjecture stated.
\end{remark}

\end{document}